\documentclass[a4paper,11pt]{article} 

\usepackage{graphicx} 
\usepackage{amsfonts}
\usepackage[hidelinks]{hyperref}
\usepackage{algorithm2e}
\usepackage[dvipsnames]{xcolor}
\usepackage{tikz} 
\usepackage{listings}
\usepackage{comment}
\usepackage{url}
\usepackage{xcolor}
\usepackage{amssymb}
\usepackage{mathtools}
\usepackage{dirtytalk}
\usepackage{float} 
\usepackage{amsmath}        	
\usepackage[utf8]{inputenc} 	
\usepackage[T1]{fontenc}    	
\usepackage{epstopdf}       	
\usepackage{enumitem}
\usepackage{mathrsfs}  
\usepackage{amsthm}

\title{Non-affine Families of 8 x 8 Complex Hadamard Matrices}

\author{
Tuomo Valtonen\\
Department of Information and Communications Engineering\\
Aalto University School of Electrical Engineering\\
P.O.\ Box 15600, 00076 Aalto, Finland\\
{\tt tuomo.valtonen@aalto.fi}}
\date{}

\begin{document}

\newcommand\numberthis{\addtocounter{equation}{1}\tag{\theequation}}

\renewcommand{\Re}{\mathop{\vcenter{\hbox{\small $\mathrm{Re}$}}}}

\renewcommand{\Im}{\mathop{\vcenter{\hbox{\small $\mathrm{Im}$}}}}

\providecommand{\keywords}[1]
{
  \small
  \noindent
  \textbf{Keywords: } #1
}

\theoremstyle{plain} 
\newtheorem{theorem}{Theorem}[section]
\newtheorem{lemma}[theorem]{Lemma}
\newtheorem{proposition}[theorem]{Proposition}
\newtheorem{conjecture}[theorem]{Conjecture}

\theoremstyle{definition} 
\newtheorem{definition}[theorem]{Definition}
\newtheorem{example}[theorem]{Example}
\newtheorem{problem}[theorem]{Problem}

\theoremstyle{remark} 
\newtheorem*{remark}{Remark} 

\maketitle

\begin{abstract}

Six non-affine 3-parameter families of complex Hadamard matrices of order 8 are
presented. These families contain Hadamard matrices that are not equivalent to any previously known Hadamard matrices in the literature. Each family arises from unimodular points of an affine variety defined by palindromic
polynomials. The families are given as an image of a function that solves the corresponding system of polynomials on a domain that guarantees unimodularity of the solutions.

\end{abstract}

\keywords{Butson-type Hadamard matrix, complex Hadamard matrix, equivalence, palindromic polynomial}

\section{Introduction}

Let $\mathbb{T}$ denote the set of unimodular complex numbers, that is, complex numbers with modulus 1. Let $ M_n(\mathbb{T})$ be the set of $n\times n$ matrices over $\mathbb{T}$.  A matrix $H\in M_n(\mathbb{T})$ is called a \textit{complex Hadamard matrix} of order $n$ if it satisfies the equation $HH^\dagger=nI$, where $H^\dagger$ is the conjugate transpose of $H$ and $I$ is the identity matrix. This is equivalent to the condition that the rows (and columns) of $H$ are orthogonal in $\mathbb{C}^n$. The set of complex Hadamard matrices of order $n$ is denoted by $\mathbb{H}(n).$

Complex Hadamard matrices have many applications in theoretical quantum physics. Examples include mutually unbiased bases (MUB)  \cite{mub2, mub1}, quantum teleportation schemes \cite{teleport} and quantum error correction \cite{ weight_quantum, errorbases}. They are also used in studies of purely mathematical subjects such as operator algebras \cite{ haagerupSet, OA1}, equiangular lines \cite{equitengular}, and spectral sets \cite{ sepctarl2, S62}. It is also believed that a better understanding of complex Hadamard matrices might shed light on the long-standing Hadamard conjecture of the existence of real Hadamard  matrices.

A matrix $B \in \mathbb{H}(n)$ whose entries are $q$th roots of unity is said to be of \textit{Butson-type}, abbreviated as BH$(n, q)$. For every order $n$, the \textit{Fourier matrix} $F_n := [\exp(\frac{2\pi i}{n}jk)]_{j,k=0}^{n-1}$, where $i$ denotes the primitive fourth root of unity, is a BH$(n,n)$ matrix. Complex Hadamard matrices often arise in continuous families depending on unimodular parameters. 
\begin{example}
\label{ex:fa}
    The matrix 
    $$F_4(a)=
    \begin{bmatrix}
         1 & 1 & 1 & 1 \\
 1 & i a & -1 & -i a \\
 1 & -1 & 1 & -1 \\
 1 & -i a & -1 & i a \\
    \end{bmatrix}$$
    is complex Hadamard for all $a\in \mathbb{T}$.
\end{example}

Complex Hadamard matrices have been classified only for orders less than $6$ (see \cite{haagerupSet}).
 Every complex Hadamard matrix of order 2, 3, and 5 is equivalent to the Fourier matrix of the corresponding order. Every complex Hadamard matrix of order 4 is equivalent to some matrix presented in Example~\ref{ex:fa}. Order 6 is conjectured to consist of (up to equivalence) two continuous families and one isolated matrix \cite{Szöllősi}. For order 7, one family \cite{P7} and multiple isolated matrices \cite{haagerupSet,Szöllősi} are known. In this article, we present six novel families of complex Hadamard matrices of order 8. Each family contains matrices that are not equivalent to any matrix previously appearing in the literature.

The outline of this paper is as follows. In Section \ref{bg_section}, we review some basic theory and definitions, generalise the concept of a palindromic polynomial to multivariate polynomials, and derive a few helpful lemmas. In Section~\ref{sec_4_prev_mat}, we list complex Hadamard matrices of order 8 already known in the literature. As we will observe, previously presented lists contain imperfections. In Section~\ref{results}, we present the new families. In Section~\ref{ineq_sec}, we establish that the families are pairwise inequivalent by leaning on the results of an exhaustive computer search. Finally, we show that there exist Butson matrices that do not fit in any of the known families.

\section{Background material}
\label{bg_section}
\subsection{Families of complex Hadamard matrices}

The general theory and applications of Hadamard matrices are covered in \cite{agaian} and \cite{horadam}. The complex case has been studied more comprehensively in \cite{Szöllősi, cgthm}.
 \begin{definition}
 
 \label{equiv}
     Two complex Hadamard matrices $H_1$ and $H_2$ are said to be \emph{equivalent} if there exist permutation matrices $P_1, P_2$ and diagonal unitary matrices $D_1, D_2$ such that
     $$H_1=D_1P_1H_2P_2D_2.$$
     When two matrices are equivalent, we write $H_1\cong H_2$. If the equality can be written so that $D_1=D_2=I$, we say that the matrices are \textit{permutation equivalent}.
 \end{definition}

A complex Hadamard matrix is called \textit{dephased} if its first row and first column consist only of ones. We define the dephase operator $\mathscr{D}$ by $\mathscr{D}(H)=D_1HD_2,$ where $D_1=\text{diag}(\bar{h}_{1,1},\bar{h}_{2,1},\ldots,\bar{h}_{n,1})$ and $D_2=\text{diag}(1,h_{1,1}\bar{h}_{1,2},\ldots,$ $h_{1,1}\bar{h}_{1,n})$. This operator turns an arbitrary matrix into a dephased one, so for every complex Hadamard matrix $H$ there exists an equivalent dephased complex Hadamard matrix $\mathscr{D}(H)$. 
We denote the set of dephased complex Hadamard matrices of order $n$ by $\mathbb{H}_d(n)$.

A \textit{neighbourhood} of $H\in M_n(\mathbb{C})$ is any open set $U\subseteq M_n(\mathbb{C})$ such that $H\in U$, where we identify $M_n(\mathbb{C})$ with $\mathbb{C}^{n^2}$. A matrix $H\in\mathbb{H}(n)$ is called \textit{isolated} if there exists a neighbourhood of $H$ that does not contain complex Hadamard matrices that are not equivalent to $H$.  In particular, none of the matrices in Example~\ref{ex:fa} are isolated, since for any fixed $a_0\in\mathbb{T}$, varying the parameter $a$ produces matrices arbitrarily close to $F_4(a_0)$ that are not equivalent to it. Fourier matrices of prime orders are all isolated \cite{P7}.  We highlight the isolation of matrix $H$ with a superscript notation $H^{(0)}$.

The isolation of a complex Hadamard matrix can be recognized by computing the corresponding \textit{defect} value introduced in \cite{cgthm} and further studied in \cite{defOfuni}. The defect of a complex Hadamard matrix $H$ is the dimension of the solution space of the real linear system obtained by linearizing the orthogonality constraints of $H$ under small entry-wise phase changes. In particular, if the defect equals $0$, the matrix is isolated.

Let $f:D\to \mathbb{H}_d(n)$ be a continuous function, where $D\subseteq \mathbb{R}^k$, and $k$ coincides with the rank of the Jacobian matrix of $f$ at some point in the interior of $D$. We call the set $H^{(k)}_{n}:=f(D)$ a \textit{k-parameter family of complex Hadamard matrices}. The Jacobian condition is imposed so that redundant parametrizations are excluded (cf. parameter change $a \to bc$ in Example~\ref{ex:fa}). The dephased form of the matrices is required so that trivial parametrizations that preserve equivalence class are excluded.  

 We say that a family of complex Hadamard matrices defined by a function $f$ is \textit{stemming} from a matrix $H\in\mathbb{H}(n)$ if $f(0,\ldots,0)=H$. A family of complex Hadamard matrices is called an \textit{affine family} if it is an image of a function $f$ in the form $f(x_1,\ldots,x_k)=[e^{ig_{j,l}}]_{j,l}$, where $g_{j,l}=c_0+c_1 x_1+\ldots +c_kx_k$ for some $c_0,\ldots c_k \in \mathbb{R}$. A family that is not affine is called a \textit{non-affine family}. In non-affine families the entry-wise mappings are more intricate. First non-affine family of complex Hadamard matrices was introduced in \cite{nonaffine}.

In practice, it is often more convenient to parametrize a family with a function $f:D\subseteq\mathbb{T}^k \to \mathbb{H}(n)$.
We will present all families in this form. They can be converted to the form used in the definition by dephasing and applying the change of variables $x_j\to e^{i\alpha_j}$ for each $j\in\{1,\ldots,k\}$. In particular, the family $F^{(1)}_4:= F_4(\mathbb{T})$ is an affine family stemming from $F_4$ as after the variable change $a\to e^{i\alpha}$ the entries are either constant or have the form $e^{i(\alpha+\pi /2)}$ or $e^{i(\alpha-\pi /2)}$.

 Let $H^{(k)}_n$ and $G^{(k)}_n$ be families of complex Hadamard matrices. If for every $A\in H^{(k)}_n$ there exists $B\in G^{(k)}_n$ such that $A\cong B$ we say that $H^{(k)}_n$ is contained in $G^{(k)}_n$. If $H^{(k)}_n$ is contained in $G^{(k)}_n$ and $G^{(k)}_n$ is contained in $H^{(k)}_n$ we say that the two families are equivalent and write $H^{(k)}_n \cong G^{(k)}_n$.

For more intuitive notation, we write that $H\in H^{(k)}_n$ also in the cases where $H$ is only equivalent to a matrix in $ H^{(k)}_n$. We define a transpose of a family with $(H_n^{(k)})^T=\{H^T\;|\;H\in H^{(k)}_n\}$. We call a family $H^{(k)}$ \textit{symmetric} if $H_n^{(k)}\cong (H_n^{(k)})^T.$

\subsection{Palindromic polynomials}

Reciprocal and palindromic polynomials are usually defined in terms of univariate polynomials, but we consider those properties for multivariate polynomials as well. In \cite{mulre}, multivariate reciprocal polynomials are considered in the context of signal processing. For the properties of univariate palindromic polynomials see \cite{palpols}.

\begin{definition}
\label{reciprocal}
    A \textit{conjugate reciprocal polynomial} $p^*$ of a multivariate polynomial $p\in\mathbb{C}[x_1,\ldots,x_n]$ is a polynomial $$p^*(x_1,\ldots,x_n)=\overline{p(\overline{x_1}^{-1},\ldots,\overline{x_n}^{-1})}.$$
\end{definition}

If the coefficients of the polynomial $p(x_1,\ldots,x_n)$ are real numbers, then Definition~\ref{reciprocal} coincides with the usual definition of reciprocal polynomial $p^*(x_1,\ldots,x_n)=$ $p(1/x_1,\ldots,1/x_n)$. If the variables are taken to be unimodular, then it coincides with the conjugate $p^*(x_1,\ldots,x_n)=\overline{p(x_1,\ldots,x_n)}$.

\begin{definition}
    A multivariate polynomial $p\in\mathbb{C}[x_1,\ldots,x_n]$ is called \textit{palindromic} if
    $$
    p=x_1^{\alpha_1}\cdots x_n^{\alpha_n}p^*
    $$
    and \textit{anti-palindromic} if
      $$
    p=-x_1^{\alpha_1}\cdots x_n^{\alpha_n}p^*,
    $$
    where $\alpha_i$ is the maximum exponent of $x_i$ in $p(x_1,\ldots,x_n)$.
\end{definition}

The symmetry of palindromic and anti-palindromic polynomials can often simplify the process of solving systems involving them.

\begin{lemma}
\label{lemma_of_palindrome}
Let $p=\sum_{j=1}^l c_j\mathbf{x_j}\in\mathbb{C}[x_1,\ldots,x_k]$ be palindromic or anti-palindromic with an even number of monomials $\mathbf{x_j}$ in lexicographic order. Then the roots of $p$ with all variables nonzero correspond to solutions of
$$\sum_{j=1}^{l/2} (c_jy_j\pm \overline{c_j}y_j^{-1})=0,$$
where a positive sign corresponds to the palindromic case, a negative sign to the anti-palindromic case and $y_j=\mathbf{x_j}/(x_1^{\alpha_1/2}\cdots x_k^{\alpha_k/2})$.    
\end{lemma}
\begin{proof}
Let $\mathbf{x}^\alpha:=x_1^{\alpha_1}\cdots x_k^{\alpha_k}$. We consider only the palindromic case, as the anti-palindromic case is analogous.

Since $p$ is palindromic,
$$
p=\sum_{j=1}^l c_j\mathbf{x}_j
=\mathbf{x}^\alpha\sum_{j=1}^l \overline{c_j}\,\mathbf{x}_j^{-1}
=\sum_{j=1}^l \overline{c_j}\,\mathbf{x}^\alpha/{\mathbf{x}_j}.
$$
Notice that the monomials $\mathbf{x}^\alpha/\mathbf{x}_1,\ldots,\mathbf{x}^\alpha/\mathbf{x}_l$ are in reverse lexicographic order. Comparing corresponding terms yields
\begin{equation}
\label{pal_relation}
c_{j}\mathbf{x}_{j}=\overline{c_{l+1-j}}\,\mathbf{x}^\alpha/{\mathbf{x}_{l+1-j}}.
\end{equation}
If all variables are assumed to be nonzero, then the roots of $p$ coincide with the solutions of
$$
\biggr (\sum_{j=1}^l c_j\mathbf{x}_j\bigg)/({x_1^{\alpha_1/2}\cdots x_k^{\alpha_k/2}})=0.
$$
Applying \eqref{pal_relation} to the terms with $j>l/2$ gives the result.
\end{proof}

Lemma~\ref{lemma_of_palindrome} also provides a method for expressing unimodular roots of palindromic polynomials with unimodular coefficients as real roots of trigonometric functions. After the change of variables $x_j \to e^{i\xi_j}$, we can employ the identities $2\cos(\xi) = e^{i\xi} + e^{-i\xi}$ and $2i\sin(\xi) = e^{i\xi} - e^{-i\xi}$. For example, finding solutions numerically can be more convenient with real numbers.

We will be interested in the unimodular roots of palindromic polynomials, and the following lemmas provide some bounds on when the obtained solution is unimodular. These two lemmas are also special cases of Cohn's theorem \cite{cohn}, since fixing all but one entry to unimodular values reduces the polynomial to a univariate \textit{self-inversive} polynomial.

\begin{lemma}
\label{pal1}
    Let $p$ be a multivariate palindromic polynomial of $n$ variables, where $\alpha_1=1$, that is,
    $$p(x_1,\ldots,x_n)=p_1(x_2,\ldots,x_n)x_1+p_0(x_2,\ldots,x_n).$$
    Fix $x_2,\ldots,x_{n}\in\mathbb{T}$. If $p(x_1,\ldots,x_n)=0$ and $p_1(x_2,\ldots,x_n)\neq 0$ then $x_1\in\mathbb{T}$.
\end{lemma}
\begin{proof} The palindromic property gives $p_1=\overline{p_0} \cdot x_2^{\alpha_2}\cdots x^{\alpha_n}_n$. Hence, solving $x_1$ gives
$x_1=-p_0/p_1 =-p_0/\overline{p_0} \cdot x_2^{-\alpha_2}\cdots x^{-\alpha_n}_n$ which has an absolute value of 1. 
\end{proof}

\begin{lemma}
\label{pal2}
    Let $p$ be a multivariate palindromic polynomial of $n$ variables, where $\alpha_1=2$, that is,
    $$p(x_1,\ldots,x_n)=p_2(x_2,\ldots,x_n)x^2_1+p_1(x_2,\ldots,x_n)x_1+p_0(x_2,\ldots,x_n).$$
    
    Fix $x_2,\ldots,x_{n}\in\mathbb{T}$. If $p(x_1,\ldots,x_n)=0$ and $p_2(x_2,\ldots,x_n)\neq 0$ then $x_1\in\mathbb{T}$ if and only if $|p_1|/|p_2| \leq 2$.
\end{lemma}
\begin{proof}
    The palindromic property gives $p_2=\overline{p_0} \cdot x_2^{\alpha_2}\cdots x^{\alpha_n}_n$ and $p_1=\overline{p_1} \cdot x_2^{\alpha_2}\cdots x^{\alpha_n}_n.$ Notice that $x_2^{\alpha_2}\cdots x^{\alpha_n}_n=e^{2\arg(p_1)i}.$ We can solve $x_1$ from
    \begin{equation}
    \label{e1}
        |p_2|e^{\arg(p_2)i}x_1^2+|p_1|e^{\arg(p_1)i}x_1+|p_2|e^{(2\arg(p_1)-\arg(p_2))i}=0.
    \end{equation}
    Division with the constant term and a variable change $x_1=ye^{(\arg(p_1)-\arg(p_2))i}$ gives
    \begin{equation}
    \label{e2}
        y^2+\frac{|p_1|}{|p_2|}y+1=0.
    \end{equation}
    Unimodular solutions of $y$ are exactly the unimodular solutions of $x_1$. We get that
    \begin{equation*}
    \label{ypal}
        y=\frac{1}{2}\left(-|p_1|/|p_2|\pm i\sqrt{4-|p_1|^2/|p_2|^2}\right),
    \end{equation*}
    which is unimodular if and only if $|p_1|/|p_2| \leq 2$.
\end{proof}

Notice that the discriminant of \eqref{e1} is the discriminant of \eqref{e2} scaled with $x_2^{\alpha_2}\cdots x^{\alpha_n}_n|p_2|^2$. In practice, it can be more convenient to consider the inequality $C/(x_2^{\alpha_2}\cdots x^{\alpha_n}_n)\leq 0$ instead of $|p_1|/|p_2| \leq 2$, where $C$ is the discriminant of \eqref{e1}. 
Moreover, Lemma~\ref{pal1} and Lemma~\ref{pal2} can be shown to work for anti-palindromic polynomials as well with almost identical proofs.

\section{Known complex Hadamard matrices of order 8}
\label{sec_4_prev_mat}
\subsection{Affine families}
In the previous literature, the following affine families of order 8 have appeared
\begin{equation}
\label{affine_list}
    F^{(5)}_{8}, D^{(5)}_{8A},\; D^{(5)}_{8B},\; S_{8A}^{(4)},\; S^{(4)}_{8B}. 
\end{equation}
The family $F_8^{(5)}:=F_8(\mathbb{T}^5)$, introduced in  \cite{cgthm}, is a 5-parameter family stemming from the Fourier matrix where $F_{8}(a,b,c,d,e)=$
$$
\exp\bigg(\frac{2 \pi i}{8}
\begin{bmatrix}
0&  0&  0&  0&  0&  0&  0&  0\\
0&  1&2&3&4&5&6&7\\
0&2&4&6&  0&2&4&6\\
0&3&6&  1&4&7&2&5\\
0&4&  0&4&  0&4&  0&4\\
0&5&2&7&4&  1&6&3\\
0&6&4&2&  0&6&4&2\\
0&7&6&5&4&3&2&  1
\end{bmatrix} \bigg) \circ
\begin{bmatrix}
1&1&1&    1&1&1&1&    1\\
1&      a&      b&          c&1&      a&      b&          c\\
1&      d&1&          d&1&      d&1&          d\\
1&      e&      b& \overline{a}ce &1&      e&      b& \overline{a}ce  \\
1&1&1&    1&1&1&1&    1\\
1&      a&      b&          c&1&      a&      b&          c\\
1&      d&1&          d&1&      d&1&          d\\
1&      e&      b& \overline{a}ce &1&      e&      b& \overline{a}ce  
\end{bmatrix},
$$
$\exp(\cdot)$ denotes element-wise exponentiation and $\circ$ is the Hadamard product. The family $D_{8A}^{(5)}:=D_{8A}(\mathbb{T}^5)$, introduced in \cite{DitaDouble},
where $D_{8A}(a,b,c,d,e)=$
$$
\exp\bigg(\frac{2 \pi i}{8}
\begin{bmatrix}
 0 & 0 & 0 & 0 & 0 & 0 & 0 & 0 \\
 0 & 1 & 2 & 3 & 4 & 5 & 6 & 7 \\
 0 & 2 & 6 & 4 & 0 & 2 & 6 & 4 \\
 0 & 3 & 6 & 1 & 4 & 7 & 2 & 5\\
 0 & 4 & 4 & 0 & 0 & 4 & 4 & 0 \\
 0 & 5 & 2 & 7 & 4 & 1 & 6 & 3 \\
  0 & 6 & 2 & 4 & 0 & 6 & 2 & 4 \\
  0 & 7 & 6 & 5 & 4 & 3 & 2 & 1 
\end{bmatrix}\bigg)  \circ
\begin{bmatrix}
1&1&1&    1&1&1&1&    1\\
1&      a&      b&          c&1&      a&      b&          c\\
1&      d&d&          1&1&      d&d&          1\\
1&      e&      b& \overline{a}ce &1&      e&      b& \overline{a}ce  \\
1&1&1&    1&1&1&1&    1\\
1&      a&      b&          c&1&      a&      b&          c\\
1&      d&d&          1&1&      d&d&          1\\
1&      e&      b& \overline{a}ce &1&      e&      b& \overline{a}ce  
\end{bmatrix}
$$
has a very similar structure to $F^{(5)}_8$, but unlike $F^{(5)}_8$, it is not symmetric. The family $D_{8B}^{(5)}$ is the transpose of the family $D_{8A}^{(5)}$. Finally, the 4-parameter family $S_{8A}^{(4)}:=S_{8A}(\mathbb{T}^4)$, introduced in \cite{S8mat}, is given by $S_{8A}(a,b,c,d)=$
    $$
    \exp\bigg(\frac{2 \pi i}{4}
\begin{bmatrix}
       0&  0&  0&  0&  0&  0&  0&  0\\
       0&  0& 2& 2& 2&  1& 3&  0\\
       0&  1&  1& 3&  0& 2& 2& 3\\
       0&  1& 3&  1& 2& 3&  1& 3\\
       0& 2& 3&  1&  0&  1& 3& 2\\
       0& 2&  1& 3& 2&  0&  0& 2\\
       0& 3& 2& 2&  0& 3&  1&  1\\
       0& 3&  0&  0& 2& 2& 2&  1\\
\end{bmatrix}\bigg)
\circ
\begin{bmatrix}
            1&       1&       1&       1& 1&       1&       1&       1\\
            1&       d&       d&       d& 1&  c d &  c d &       d\\
            1&  a \overline{d} &  b \overline{d} &  b \overline{d} & 1&       1&       1&  a \overline{d} \\
            1&       a&       b&       b& 1&  c d &  c d &       a\\
            1&       1&  b \overline{d} &  b \overline{d} & 1&       c&       c&       1\\
            1&       d&       b&       b& 1&       d&       d&       d\\
            1&  a \overline{d} &       1&       1& 1&       c&       c&  a \overline{d} \\
            1&       a&       d&       d& 1&       d&       d&       a
\end{bmatrix}
$$ and its transpose is denoted by $S_{8B}^{(4)}\cong(S_{8A}^{(4)})^T.$
The family $S_{8A}^{(4)}$ is often referred to without the subscript $A$ and its transpose is sometimes called $D^{(4)}_8$. 

To our knowledge, the pairwise relationships among these families have not all been explored. At least families $D_{8A}^{(5)}$ and $D_{8B}^{(5)}$ are missing from some of the earlier listings of affine families of order 8.

\subsection{Isolated matrices}

Four isolated complex Hadamard matrices of order 8 are known, namely
\begin{equation}
\label{isolated_list}
   A_{8A}^{(0)},\;A^{(0)}_{8B},\; V_{8A}^{(0)},\;V_{8B}^{(0)}.
\end{equation}
The isolation of the matrices in \eqref{isolated_list} follows from their defect values, since each matrix has defect $0$.

The matrix
$$
A_{8A}^{(0)}=
\begin{bmatrix}
 a & 1 & -a & 1 & -a & -a & a
   & 1 \\
 1 & -a & 1 & a & -a & a & -a
   & 1 \\
 -a & 1 & -a & a & 1 & a & 1 &
   -a \\
 1 & a & a & 1 & 1 & -a & -a &
   -a \\
 -a & -a & 1 & 1 & a & -a & 1
   & a \\
 -a & a & a & -a & -a & 1 & 1
   & 1 \\
 a & -a & 1 & -a & 1 & 1 & a &
   -a \\
 1 & 1 & -a & -a & a & 1 & -a
   & a \\
\end{bmatrix},
$$ where $a$ is the root of the polynomial $3x^2-2x+3$, is introduced in \cite{bruzda8}. The matrix $A^{(0)}_{8A}$ is symmetric but not equivalent to its conjugate, which is denoted by $A_{8B}^{(0)}$.

Consider a matrix-valued function 
$$V_8(a,b,c)=
\begin{bmatrix}
 -a b & -a b & b c & b c & 1 & 1 & a c & a c \\
 -a b & b c & -a b & 1 & b c & a c & 1 & -a c \\
 b c & -a b & 1 & -a b & a c & b c & -a c & 1 \\
 b c & 1 & -a b & a c & -a b & -a c & b c & -1 \\
 1 & b c & a c & -a b & -a c & -a b & -1 & b c \\
 1 & a c & b c & -a c & -a b & -1 & -a b & -b c \\
 a c & 1 & -a c & b c & -1 & -a b & -b c & -a b \\
 a c & -a c & 1 & -1 & b c & -b c & -a b & a b \\
 \end{bmatrix}.$$
The matrices $V_{8A}^{(0)}$ and $V_{8B}^{(0)}$ are isolated matrices, presented in \cite{Catalogue} and analytically solved later in \cite{bruzdaSH}. These matrices are found by solving the system of polynomial equations arising from orthogonality constraints of the matrix $$V_8(\frac{\sqrt{a}}{\sqrt{b} \sqrt{c}},\frac{\sqrt{b}}{\sqrt{a} \sqrt{c}}, \frac{\sqrt{a} \sqrt{b}}{\sqrt{c}}).$$ In \cite{bruzdaSH}, analytic solutions for these equations are obtained from the roots of a polynomial of degree 16, whose roots can be expressed in terms of the roots of quartic polynomials. We present an alternative solution which leads to a more compact way of expressing the solutions with trigonometric functions. In some previous literature, all the solutions have been treated as different matrices, even though there are only two that are inequivalent.
 
\begin{lemma}
\label{Vlemma}
    $V_8(a,b,c)$ is a complex Hadamard matrix if and only if
\begin{equation*}
\begin{cases}
    & \Re(a)\Re(c)-\Im(a)\Im(b) = 0\\
    & \Re(a)\Re(b) - \Im(a)\Im(c) + \Re(b)\Re(c) = 0\\
    & \Re(a)\Re(b) + \Im(a)\Im(c) + \Im(b)\Im(c) = 0
\end{cases}
\end{equation*}
where $a,b,c\in\mathbb{T}$.
\end{lemma}
\begin{proof}
    Orthogonality requirements between the rows of $V_8(a,b,c)$ give us a system of polynomial equations $p_V^{(1)}=0$, $p_V^{(2)}=0$ and $p^{(3)}_V=0$, where
    \begin{align*}
        &p_V^{(1)}=b + a^2 b + c - a^2 c - b^2 c + a^2 b^2 c + b c^2 + a^2 b c^2,\\
        &p_V^{(2)}=a + b - a^2 b - b c^2 + a^2 b c^2 + a b^2 c^2,\\
        &p^{(3)}_V=a b^2 + c + a^2 c + b^2 c + a^2 b^2 c + a c^2.
    \end{align*}

We can equivalently consider a system $p^{(1)}_V=0$, $p^{(2)}_V+p_V^{(3)}=0$ and $p^{(2)}_V-p_V^{(3)}=0$. Notice that $p^{(1)}_V$, $p^{(2)}_V+p_V^{(3)}$ and $p^{(2)}_V-p_V^{(3)}$ are palindromic. Consider the constraint $p^{(2)}_V+p_V^{(3)}=0$. Lemma~\ref{lemma_of_palindrome} gives an equivalent constraint
$$a b+\frac{1}{a b}+\frac{a}{b}+\frac{b}{a}+a c+\frac{1}{a c}-\frac{a}{c}-\frac{c}{a}+b c+\frac{1}{b c}+\frac{c}{b}+\frac{b}{c}=0,$$
which factors to 
$$\left(a+\frac{1}{a}\right) \left(b+\frac{1}{b}\right)+\left(a-\frac{1}{a}\right) \left(c-\frac{1}{c}\right)+\left(b+\frac{1}{b}\right) \left(c+\frac{1}{c}\right)=0$$
and by unimodularity of $a,b,c$ we have
$$\Re(a)\Re(b) - \Im(a)\Im(c) + \Re(b)\Re(c) = 0.$$
Other constraints can be dealt with in a similar manner.
\end{proof}

\begin{proposition}
Let
{\small
\begin{alignat*}{2}
    &x_1=\sqrt{\frac{1}{2} \sin \left(\frac{\pi }{16}\right) \sec \left(\frac{3 \pi }{16}\right)},\quad
    && x_2=\sqrt{\cos \left(\frac{\pi }{16}\right) \cos \left(\frac{3 \pi }{16}\right) \sec \left(\frac{\pi }{8}\right)},
\end{alignat*}
}
\vspace{-0.2cm}
{\small
\begin{align*}
&y_1=2 \cos \left(\frac{\pi }{16}\right) \sqrt{\sin \left(\frac{\pi }{16}\right) \sec \left(\frac{3 \pi }{16}\right)},\\
& y_2= -\sin \left(\frac{\pi }{16}\right)\sqrt{2 \cos \left(\frac{\pi }{16}\right) \sec \left(\frac{\pi }{8}\right) \sec \left(\frac{3 \pi }{16}\right)} ,\\
&z_1=-2 \cos \left(\frac{\pi }{16}\right) \sec \left(\frac{\pi }{8}\right) \sqrt{\sin \left(\frac{\pi }{16}\right) \cos \left(\frac{3 \pi }{16}\right)},\\
& z_2=-\frac{1}{2} \sec\left(\frac{\pi }{8}\right)\sqrt{\cos \left(\frac{3 \pi }{16}\right) \sec \left(\frac{\pi }{16}\right) \sec\left(\frac{\pi }{8}\right)}.
\end{align*}
}
$V_8(x_1+x_2i,\;y_1+y_2i,\;z_1+z_2i)$ is a complex Hadamard matrix.
\end{proposition}
\begin{proof}
    The given assignment is a solution for the system of Lemma~\ref{Vlemma}. 
\end{proof}

We define $V_{8B}^{(0)}:=V_8(x_1+x_2i,\;y_1+y_2i,\;z_1+z_2i)$ and $V_{8A}^{(0)}:= (V_{8B}^{(0)})^\dagger$. Let $P_1=(e_4\;e_3\; e_7\;e_8\;e_1\;e_2\;e_6\;e_5)$ and $P_2=(e_1\;e_7\;e_6\;e_4\;e_5\;e_3\;e_2\;e_8)$. The matrices $\mathscr{D}(V_{8A}^{(0)})$ and $\mathscr{D}(P_1V_{8B}^{(0)}P_2)$ match exactly the numerical matrices of the same name given in \cite{Catalogue}.

All the 16 solutions to the system of Lemma~\ref{Vlemma} that do not produce Butson matrices have the form 
$$(a,b,c)=(\mu_1\cdot x_1+\mu_2 \cdot x_2i,\;\mu_3 \cdot y_1+\mu_2 \cdot y_2i,\;\mu_1 \cdot z_1+\mu_1\mu_2\mu_3 \cdot z_2i)$$ 
or 
$$(a,b,c)=(\mu_1 \cdot x_2+\mu_2 \cdot x_1i,\;\mu_3 \cdot z_2+\mu_2 \cdot z_1i,\;\mu_1\cdot y_2+\mu_1\mu_2\mu_3\cdot y_1i),$$
where $\mu_1,\mu_2,\mu_3\in \{-1,1\}$. These can all be shown to produce matrices that are equivalent to either $V_{8B}^{(0)}$ or $V_{8A}^{(0)}$. The numerical values for the different solutions given in \cite{Catalogue} are a perfect match to the approximate values of $(c\overline{b}, c\overline{a}, \overline{ab})$, where $(a,b,c)$ is one of the solutions above.

\subsection{Non-affine families}
The existing literature also knows two non-affine families:
\begin{equation}
    \label{non-affine_list}
    T^{(1)}_8, \;T_{8B}^{(3)}.
\end{equation}
The first is a 1-parameter family introduced in \cite{bruzda8}. The family is given by the function $T_8(x,y,z,u)=$

$$
\begin{bmatrix}
1 & 1 & 1 & 1 & 1 & 1 & 1 & 1 \\
1 & x & y & z & -yz\overline{x} & \zeta^3 & \zeta^{15} & \zeta^{18} \\
1 & -u\overline{y}\zeta^8 & u\overline{x}\zeta^8 & -uz\overline{x}\overline{y}\zeta^8 & -uz\overline{x}^{\,2}\zeta^8 & \zeta^{13} & \zeta^{15} & \zeta^8 \\
1 & y\zeta^2 & x\zeta^2 & -yz\overline{x}\zeta^2 & z\zeta^2 & \zeta^{17} & \zeta^5 & \zeta^2 \\
1 & u\overline{x} & -u\overline{y} & -uz\overline{x}^{\,2} & -uz\overline{x}\overline{y} & \zeta^7 & \zeta^5 & \zeta^{12} \\
1 & -ix\overline{z} & ix\overline{z} & -iz\overline{x} & iz\overline{x} & 1 & \zeta^{10} & \zeta^{10} \\
1 & u & -u & -uz^2\overline{x}^{\,2} & uz^2\overline{x}^{\,2} & \zeta^{10} & 1 & \zeta^{10} \\
1 & iuz\overline{x} & iuz\overline{x} & -iuz\overline{x} & -iuz\overline{x} & \zeta^{10} & \zeta^{10} & 1
\end{bmatrix}
$$
 where $\zeta = \exp(2\pi i/20)$. The domain where the function $T_8(x,y,z,u)$ gives Hadamard matrices can be expressed solely in terms of $x$. The relations of the other variables and $x$ are complicated and thus not presented here. The parameter $x$ is not free on $\mathbb{T}$, and the exact bounds that guarantee the unimodularity of the other variables are not known. We will show in Section \ref{results} that this is most likely a subfamily of a new family we have discovered.

The non-affine 3-parameter family $T^{(3)}_{8B}$ is introduced in \cite{bruzdaSH} and given by the function $T_{8B}(a,b,c,d)=$
\begin{equation*}
\label{TB}
\begin{bmatrix}
 1 & 1 & 1 & 1 & 1 & 1 & 1 & 1 \\
 1 & -1 & -d & -c & d & c & c d & -c d \\
 1 & -a & -1 & c & a & -a c & -c & a c \\
 1 & a & d & -1 & a d & -a & -d & -a d \\
 1 & b & -b & -b c & -1 & -c & c & b c \\
 1 & -b & b d & b & -d & -1 & d & -b d \\
 1 & -a b & b & -b & -a & a & -1 & a b \\
 1 & a b & -b d & b c & -a d & a c & -c d & -a b c d \\
\end{bmatrix},
\end{equation*}
with some additional constraints for $d$.
As the family is discussed only briefly in \cite{bruzdaSH}, we formulate the following more detailed result.

\begin{proposition}
\label{prop33}
    Let $T_{8B}^{(3)}:=\mathcal{S}_{B}(\mathbb{T}^3\backslash D_0)$, where
    $$\mathcal{S}_B(a,b,c)=T_{8B}(a,b,c,d')$$
    with
    $$d'=\frac{1 + a b + a c + bc}{a + b + c + a b c }$$
    and $D_0=\{(a,b,c)\in \mathbb{T}^3\;|\; -1,1\in\{a,b,c\}\}$. $T_{8B}^{(3)}$ is a symmetric 3-parameter family of complex Hadamard matrices. Additionally, for every $H\in T_{8B}(\mathbb{T}^4)\cap\mathbb{H}(8)$ there exists an equivalent matrix in $T^{(3)}_{8B}.$
\end{proposition}

\begin{proof}
Each orthogonality constraint between the rows of $T_{8B}$ simplifies to 0 or
\begin{equation}
\label{ortTB8}
    p_B=1 + a b + a c + bc - d(a + b + c + a b c )=0,
\end{equation}
from which $d$ can be solved with respect to three other variables. We can notice that $p_B$ is anti-palindromic. By Lemma~\ref{pal1}, the obtained solution for $d$ gives unimodular values for any $a,b,c\in\mathbb{T}$ as long as the denominator of $d'$ is not 0. The denominator is 0 exactly on the set $D_0$.

Consider permutation matrices  $P_1=(e_1\;e_2\;e_4\;e_6\;e_3\;e_5\;e_7\;e_8)$ and $P_2=(e_1\;e_2\;e_5\;e_3\;e_6\;e_4\;e_7\;e_8)$. We have that $\mathcal{S}_B(-d',c,a)=P_1\mathcal{S}_B^TP_2(a,b,c)$ implying that $T^{(3)}_{8B}$ is symmetric. Additionally, one can find similar assignments of $\mathcal{S}_B$ that show that the Hadamard matrices in $T_{8B}(D_0)$ are all equivalent to some matrix in $T^{(3)}_{8B}$. 
\end{proof}

In \cite{bruzdaSH}, a third non-affine family called $T^{(3)}_{8C}$ is also introduced. We have left this family out of our considerations as the relationships of the entries are not fully solved, and its existence is observed by utilising numerical methods. We have, however, found and solved a family that seems to contain all the matrices that are obtained from the given numerical construction of $T^{(3)}_{8C}$. We will be referring to it with the same name.

\section{Novel matrix families}
\label{results}

In this section, we introduce four new families of complex Hadamard matrices. All the families are given as the image of a matrix-valued function on a specific domain. 

To be precise, we should show that there exist points in the interior of those domains at which the rank of the Jacobian matrix equals the number of parameters. We have left that to Section~\ref{in_seq}, where we also show that there are points in the domains that produce matrices that do not belong to any of the other families. 

Each of the families to be presented is defined with a matrix-valued function on unimodular variables $a,$ $b$ and $c$. The families can be recognized as non-affine as in each of them there are entries that do not simplify to a product of the form $s a^xb^yc^z$ for $s\in\mathbb{T}$ and $x,y,z\in \mathbb{R}$, which would correspond to an affine transformation in the argument of the entry. 

\subsection{Family $T_{8C}^{(3)}$}

Consider a matrix-valued function $T_{8C}(a,b,c,d,e,f)=$
$$
\begin{bmatrix}
 a & b d & d & a b & a d & b c & a c & b d \\
 b d & -a & a b & -d & b c & -a d & b d & -a c \\
 d & a b & b d e f & -a e f & a c f & b d f & b c e & -a d e \\
 a b & -d & -a e f & -b d e f & b d f & -a c f & -a d e & -b c e \\
 a c & b d & a d f & b c f & -c d f & -a b c f & -a c & -b c d \\
 b d & -a c & b c f & -a d f & -a b c f & c d f & -b c d & a c \\
 a d & b c & b d e & -a c e & -a c & -b c d & -a b c e & c d e \\
 b c & -a d & -a c e & -b d e & -b c d & a c & c d e & a b c e \\
\end{bmatrix}
.$$

\begin{lemma}
\label{TClemma}
$T_{8C}(a,b,c,d,e,f)$ is a complex Hadamard matrix if and only if 
$p^{(1)}_C=0$, $p^{(2)}_C=0$ and $p^{(3)}_C=0$, where
\begin{align*}
p^{(1)}_C=&a^2 c-a^2 b^2 c+a^2 c^2-b^2 c^2+a^2 d^2-b^2 d^2+c d^2-b^2 c d^2,\\
p^{(2)}_C=&a^2 b c+b c d^2-a^2 c d e+b^2 c d e+a b c^2 f+a b d^2 f-a c d e f+a b^2 c d e f,\\ 
p^{(3)}_C=&a c d-a b^2 c d+a b c^2 e+a b d^2 e+a^2 c d f-b^2 c d f+a^2 b c e f+b c d^2 e f
\end{align*}
and $a,b,c,d,e,f\in\mathbb{T}.$

\end{lemma}
\begin{proof}
We need that $a,b,c,d,e,f\in\mathbb{T}$ for unimodular entries. Now we can identify complex conjugates of the variables with their multiplicative inverses. Each inner product between two distinct rows of $T_{8C}$, that is not equal to 0, can be factored to $p_{C}^{(j)}/q$ for some $j\in\{1,2,3\}$, where $q$ is a monomial on variables $a,b,c,d,e,f$.
\end{proof}

We can obtain a 3-parameter family of complex Hadamard matrices by solving the system of polynomial equations in Lemma~\ref{TClemma}.

\begin{theorem}
\label{TC_theorem}
    Let $T^{(3)}_{8C}:=S_C(D)$, where
    $$\mathcal{S}_C(a,b,c)=T_{8C}(a,b,c,d',e',f'),$$
    with 
    \begin{align*}
    d'=&-\sqrt{c A_1}/\sqrt{B_1},\\
    e'=&d' ((1 - a) (1 + a) (a - b^2) (a + b^2) A_3 + 
    \sqrt{C})/B_3,\\
    f'=&((c-b^2) (c+1) (a^2+c) (a^2-b^2 c)A_2+\sqrt{C})/B_2,
    \end{align*}
    where
    \begin{align*}
        A_1=&-a^2 +a^2 b^2 -a^2 c+b^2 c,\\
        A_2=&a^2+a^4-4 a^2 b^2+b^4+a^2 b^4,\\
        A_3=&-a^4 b^2-a^4 c+a^2 b^2 c+a^4 b^2 c-a^2 b^4 c-a^2 c^2-a^4 c^2+6 a^2 b^2 c^2\\
        &-b^4 c^2-a^2 b^4 c^2-a^2 c^3+b^2 c^3+a^2 b^2 c^3-b^4 c^3-b^2 c^4,\\
        B_1=&a^2-b^2+c-b^2 c, \\
        B_2=&-2 a (a-b) (-1+b) (1+b) (a+b) A_3,\\
        B_3=& 2 bc (a-c)(a+c)A_1A_2,\\
        C=&(1+c)^2 (a^2+c)^2 (-b^2+c)^2 (a^2-b^2 c)^2A_2^2 -4 a^2 (a^2-b^2)^2 (-1+b^2)^2 A_3^2
    \end{align*}
    and $$D=\{(a,b,c)\in\mathbb{T}^3\;|\;C/(a^8b^8c^4) \leq0,\;B_1,B_2,B_3\neq 0\}.$$ $T^{(3)}_{8C}$ is a 3-parameter family of complex Hadamard matrices.
\end{theorem}
\begin{proof}
 One can check that the assignment $d=d'$, $e=e'$, and $f=f'$ satisfy the system in Lemma~\ref{TClemma}. We can notice that $e'$ and $f'$ are roots of a palindromic polynomial, so by Lemma~\ref{pal2} they get unimodular values whenever $a,b,c,d'\in \mathbb{T}$ and $C/(a^8b^8c^4)\leq 0$. Also, $d'$ is a root of an anti-palindromic polynomial, and it gets unimodular values whenever $a,b,c\in \mathbb{T}$.
\end{proof}

\subsection{Family $T_{8D}^{(3)}$}
Consider a matrix-valued function $T_{8D}(a,b,c,d,e,f)=$
$$
\begin{bmatrix}
  1 & -c & c & -e & -c & e & -e & c e \\
 1 & c & -e & c & -e & c & -e & -c e \\
 c & -1 & c & e & -c & -e & c e & -e \\
 c & 1 & e & c & e & c & c e & e \\
 b & f & a d & a b & d f & b f & a b d & a d f \\
 b & -f & a b & a d & -b f & -d f & a b d & -a d f \\
 f & b & -a d & a b & -d f & b f & -a d f & -a b d \\
 f & -b & a b & -a d & -b f & d f & -a d f & a b d \\
\end{bmatrix}
.$$

\begin{lemma}
    \label{TDlemma}
    $T_{8D}(a,b,c,d,e,f)$ is a complex Hadamard matrix if and only if $p^{(1)}_D=0$, $p^{(2)}_D=0$, $p^{(3)}_D=0$ and $p^{(4)}_D=0$, where
    \begin{align*}
p^{(1)}_D=&a b c+a b c d-b c e-a d e-b c f-a d f+e f+d e f,\\
p^{(2)}_D=&\;a c d+a b c d-a b e-b c e+a d f+c d f-e f-b e f,\\ 
p^{(3)}_D=&\;a b c+a b d+b e+a d e-b c f-a c d f-c e f-d e f,\\
p^{(4)}_D=&\;a b d+a c d+b e+a b e+c d f+a c d f+b e f+c e f,
\end{align*}
and $a,b,c,d,e,f\in\mathbb{T}.$
\end{lemma}
\begin{proof}
      We need that $a, b, c, d, e, f\in\mathbb{T}$ for unimodular entries. Similarly to Lemma~\ref{TClemma}, the system of equations that requires orthogonality of the rows can be reduced to $p^{(1)}_D=0$, $p^{(2)}_D=0$, $p^{(3)}_D=0$ and $p^{(4)}_D=0$.
\end{proof}

\begin{theorem}
    Let $T^{(3)}_{8D}:=\mathcal{S}_D(D),$ where
    $$\mathcal{S}_D(a,b,c)=T_{8D}(a,b,c,d',e',f')$$
    with
    \begin{align*}
    &d'=-A_1/B_1,\\
    &e'=c ((a - b) (a b - c) (1 + c) + \sqrt{C})/B_2,\\
    &f'=-((1 + c)A_2 + \sqrt{C})/(2 B_3),
    \end{align*}
    where
    \begin{align*}
        A_1&=b c+a b c+b^2 c+b c^2,\\
        A_2&=2 a b+a^2 b+a b^2+a c+b c+2 a b c,\\
        B_1&=a b+a c+b c+a b c,\\
        B_2&=2 (c-1)B_1,\\
        B_3&=a b+c+2 a c+2 b c+a b c+c^2,\\
        C&=(1+c)^2 A_2^2-4 a b B_3^2,
    \end{align*}
    and $$D=\{(a,b,c)\in\mathbb{T}^3\;|\; C/(a^2b^2c^2)\leq 0,\;B_1,B_2,B_3\neq0\}.$$ $T^{(3)}_{8D}$ is a 3-parameter family of complex Hadamard matrices.
\end{theorem}
\begin{proof}
    One can check that the assignment $d = d'$, $e = e'$, and $f = f'$
satisfy the system in Lemma~\ref{TDlemma}. Additionally, $d'$, $e'$ and $f'$ are all roots of palindromic polynomials. By Lemma~\ref{pal1} and Lemma~\ref{pal2}, $d'$, $e'$ and $f'$ are unimodular whenever $a,b,c\in\mathbb{T}$ and $C/(a^2b^2c^2)\leq 0$.
\end{proof}

In \cite{bruzdaSH}, it was suggested that the family $T^{(1)}_8$ could be contained in $T_{8B}^{(3)}$ or $T_{8C}^{(3)}$. Testing this with the algorithm given in Section~\ref{alg} shows that this is not the case. However, the following propositions show that the transpose of $T^{(1)}_8$ satisfies the system of Lemma~\ref{TDlemma}.

\begin{proposition}
\label{T1subfam}
    For all $H\in T^{(1)}_8$ there exists $H'\in T_{8D}(\mathbb{T}^6)$ such that $H^T\cong H'$.
\end{proposition}
\begin{proof}
    Consider permutation matrices
        $$P_1=(e_1\;e_6\;e_5\;e_8\;e_7\;e_4\;e_3\;e_2)\text{ and }P_2=(e_1\;e_6\;e_3\;e_5\;e_2\;e_4\;e_8\;e_7).$$
We have that
 $$\mathscr{D}(T_{8D})(e^{i\pi\frac{3}{10}}u/y,\;iz,\;i,\;-iyz/x,\;e^{i\pi\frac{3}{10}},\;ix)=P_1T_8^TP_2(x,y,z,u).$$

\end{proof}

\subsection{Family $T_{8E}^{(3)}$}
Consider a matrix valued function $T_{8E}(a,b,c,d,e,f)=$
$$
\begin{bmatrix}
1 & 1 & 1 & 1 & 1 & 1 & 1 & 1 \\
 1 & d & e & -f & d f & e f & -d e & d e f \\
 1 & a & -b & f & -a f & b f & a b & a b f \\
 1 & -d & b & c & c d & -b c & b d & b c d \\
 1 & -a & -e & -c & -a c & -c e & -a e & a c e \\
 1 & -a d & b e & -1 & a d & -b e & -a b d e & a b d
   e \\
 1 & a d & -1 & -c f & -a c d f & c f & -a d & a c d
   f \\
 1 & -1 & -b e & c f & -c f & -b c e f & b e & b c e
   f \\
\end{bmatrix}.
$$

\begin{lemma}
\label{TElemma}
    $T_{8E}(a,b,c,d,e,f)$ is a complex Hadamard matrix if and only if $p^{(1)}_E=0$, $p^{(2)}_E=0$, $p^{(3)}_E=0$ and $p^{(4)}_E=0$, where
        \begin{align*}  
        &p^{(1)}_E=1+b+c-b c-d+b d+c d+b c d,\\
        &p^{(2)}_E= 1 - a - c - a c - e - a e - c e + a c e,\\
        &p^{(3)}_E= 1 + a - b + a b + f - a f + b f + a b f,\\
        &p^{(4)}_E=1+d+e-d e-f+d f+e f+d e f,
        \end{align*}
    and $a,b,c,d,e,f\in\mathbb{T}.$
\end{lemma}
\begin{proof}
    Similar to Lemma~\ref{TClemma} and Lemma~\ref{TDlemma}, any constraint requiring orthogonality of the rows of $T_{8E}(a,b,c,d,e,f)$ can be factored into one of the equations above.
\end{proof}

As each of the equations in Lemma~\ref{TElemma} contains only three of the six variables, the system is straightforward to solve.

\begin{theorem}
\label{TEtheorem}
   Let $T_{8E}^{(3)}:=\mathcal{S}_{E}( \mathbb{T}^3\backslash D_0),$ where
   $$\mathcal{S}_E(a,b,c)=T_{8E}(a,b,c,d',e',f')$$ with
\begin{align*}
 &d'=(-1-b-c+b c)/(-1+b+c+b c),\\
 &e'=(-1+a+c+a c)/(-1-a-c+a c),\\
 &f'=(-1-a+b-a b)/(1-a+b+a b)
\end{align*}
and $D_0=\{(s,s,t),\ (s,t,-s),\ (t,s,-s)\;|\;
s\in\{i,-i\},\ t\in\mathbb{T}\}$.
$T^{(3)}_{8E}$ is a 3-parameter family of complex Hadamard matrices.
\end{theorem}
\begin{proof}
Consider the system of polynomial equations given in Lemma~\ref{TElemma}. Variables $d,\ e$ and $f$ can be solved with respect to $a,\ b$ and $c$ from the equations $p_E^{(1)}=0,\;p_E^{(2)}=0$ and $p_E^{(3)}=0$. The obtained solutions satisfy $p_E^{(4)}=0$ as well. The points where at least one denominator vanishes form exactly the set $D_0$. By Lemma~\ref{pal1}, $d'$, $e'$ and $f'$ get unimodular values whenever $(a,b,c)\in\mathbb{T}^3 \backslash D_0$.
\end{proof}

\subsection{Family $T^{(3)}_{8F}$}

Consider a matrix-valued function $T_{8F}(a,b,c,d,e,f,g,h)=$
$$
\begin{bmatrix}
 1 & a b & e f & a b e f & c d & d h & c g & g h \\
 a b & 1 & -a b e f & -e f & d h & c d & -g h & -c g \\
 e f & -a b e f & -1 & a b & -c g & g h & c d & -d h \\
 a b e f & -e f & a b & -1 & -g h & c g & -d h & c d \\
 b e & -a e & b f & -a f & 1 & -c h & d g & -c d g h \\
 a e & -b e & -a f & b f & c h & -1 & -c d g h & d g \\
 b f & a f & -b e & -a e & -d g & -c d g h & 1 & c h \\
 a f & b f & a e & b e & -c d g h & -d g & -c h & -1 \\
\end{bmatrix}.
$$

\begin{lemma}
\label{TFlemma}
     $T_{8F}(a,b,c,d,e,f,g,h)$ is a complex Hadamard matrix if and only if $p^{(1)}_F=0$, $p^{(2)}_F=0$, $p^{(3)}_F=0$, $p^{(4)}_F=0$, $p^{(5)}_F=0$ and $p^{(6)}_F=0$, where
        \begin{align*}  
        p^{(1)}_F=&-c d+b^2 c d+b e-b c^2 e+b d^2 e-b c^2 d^2 e-c d e^2+b^2 c d e^2,\\
        p^{(2)}_F=&-a e+a d^2 e+d h-a^2 d h-d e^2 h+a^2 d e^2 h+a e h^2-a d^2 e h^2,\\
        p^{(3)}_F=& \ b f+b c^2 f-c g-b^2 c g+c f^2 g+b^2 c f^2 g-b f g^2-b c^2 f g^2,\\
        p^{(4)}_F=&-a f-a f g^2+g h+a^2 g h+f^2 g h+a^2 f^2 g h-a f h^2-a f g^2 h^2,\\
        p^{(5)}_F=&\ a b c^2 d^2 e f+a b c^2 e f g^2-c d g h+a^2 b^2 c d g h-c d e^2 f^2 g h\\&+a^2 b^2 c d e^2 f^2 g h-a b d^2 e f h^2-a b e f g^2 h^2,\\
        p^{(6)}_F=&\ a b e f+a b d^2 e f g^2+a^2 c d e^2 g h-b^2 c d e^2 g h+a^2 c d f^2 g h\\&-b^2 c d f^2 g h-a b c^2 e f h^2-a b c^2 d^2 e f g^2 h^2,
        \end{align*}
    and $a,b,c,d,e,f,g,h\in\mathbb{T}$.    

\end{lemma}
\begin{proof}
     All the orthogonality requirements can be factored to depend only on the given polynomials.
\end{proof}

We can notice that assignment $(a,b,c,d,e,f,g,h)=(a,b,b,e,e,f,f,a)$ satisfy the equations $p^{(1)}_F=0$, $p^{(2)}_F=0$, $p^{(3)}_F=0$ and $p^{(4)}_F=0$. Additionally, $-p^{(5)}_F(a,b,b,e,e,f,f,a)=p^{(6)}_F(a,b,b,e,e,f,f,a)=a b e f\cdot p_B(a^2,e^2,f^2,b^2)$. Thus, we are left with the orthogonality constraint $p_B(a^2,e^2,f^2,b^2)=0$, which can be solved in the same manner as in Proposition~\ref{prop33}. Introducing suitable variable changes, we can notice that the matrices given by this solution are all equivalent to matrices in $T_{8B}^{(3)}$.

There are also matrices we have not seen before in $T_{8F}(\mathbb{T}^8)$ and we need the following lemma to find them.

\begin{lemma}
\label{TFlemma2}
The system of equations given in Lemma~\ref{TFlemma} is equivalent to the system
     \begin{equation}
 \label{F2}
\begin{cases}
    \Im(b)\Re(e)-\Im(c)\Re(d)=0\\
    \Im(a)\Im(e)-\Im(d)\Im(h)=0\\
    \Re(b)\Im(f)-\Re(c)\Im(g)=0\\
    \Re(a)\Re(f)-\Re(g)\Re(h)=0\\
    \Im(ab)\Re(ef)-\Re(d\overline{g})\Im(h\overline{c})=0\\
    \Im(a\overline{b})\Re(e\overline{f})-\Re(dg)\Im(hc)=0
\end{cases},
\end{equation}
where $a,b,c,d,e,f,g,h\in\mathbb{T}.$
\end{lemma}
\begin{proof}
    The polynomials $p^{(i)}_F$ are palindromic or anti-palindromic for all $i\in\{1,\ldots,6\}$. By Lemma~\ref{lemma_of_palindrome}, we can write $p_F^{(1)}=0$ as 
    $$b e-\frac{1}{b e}+\frac{b}{e}-\frac{e}{b}+\frac{d}{c}-\frac{c}{d}+\frac{1}{c d}-c d=0$$
    This can be factored to 
    $$\left(b-\frac{1}{b}\right) \left(e+\frac{1}{e}\right)-\left(c-\frac{1}{c}\right) \left(d+\frac{1}{d}\right)=0$$
    and by unimodularity of the variables, this is equivalent to
    $$\Im(b)\Re(e)-\Im(c)\Re(d)=0.$$
    The other constraints can be dealt with in a similar manner.
\end{proof}

We can solve the system \eqref{F2} by introducing different variables for real and imaginary parts and adding unimodularity constraints between them.

\begin{theorem}
\label{TF}
    Let $T_{8F}:=\mathcal{S}_F(D)$, where
    $$\mathcal{S}_F(a_r,b_r,c_i)=T_{8F}(a_r+ia_i,b_r+ib_i,c_r+ic_i,d',e',f',g',h'),$$
    with
    \begin{alignat*}{3}
        &a_i=\sqrt{1-a_r^2},\;\;\; && b_i=\sqrt{1-b_r^2},  &&c_r=\sqrt{1-c_i^2},\\
        &d'=d_r+id_i,&&d_r=\frac{b_i \sqrt{ A_2}}{c_i\sqrt{ B_1B_2}}, &&d_i=i\frac{a_i \sqrt{A_1B_3}}{|c_i|  \sqrt{B_1B_2}},\\
        &e'=e_r+ie_i,&&e_r=\frac{ \sqrt{A_2}}{\sqrt{B_1B_2} },\quad&&e_i=i\frac{a_r b_i \sqrt{A_1}}{\sqrt{B_1B_2 }},\\
        &f'=f_r+if_i,&&f_r=-\frac{a_ib_r|c_i|\sqrt{A_1} } { c_i\sqrt{B_1B_4} },\;\; &&f_i=i\frac{|c_i|\sqrt{A_2}}{c_i\sqrt{B_1B_4}},\\
        &g'=g_r+ig_i,&&g_r=\frac{a_r\sqrt{A_1B_3}}{c_r\sqrt{B_1B_4}}, &&g_i=i\frac{b_r|c_i|\sqrt{A_2}}{c_r c_i\sqrt{B_1B_4}},\\
        &h'=h_r+ih_i,&&h_r=-\frac{a_i b_r c_r|c_i|}{\sqrt{B_3}c_i}, &&h_i=\frac{ a_r b_i |c_i|}{\sqrt{B_3}},
    \end{alignat*}
    where
    \begin{alignat*}{2}
        &A_1=(c_i - b_i) (b_i + c_i), && A_2=-a_i^4 b_r^2+a_r^4 c_i^2+a_i^2 b_r^2 c_i^2-a_r^2 b_r^2 c_i^2,\\
        &B_1=(a_r - b_r) (a_r + b_r),&&B_2=-a_r^2 b_r^2 + a_r^2 c_i^2 + c_r^2,\\
        &B_3=a_i^2 b_r^2 + a_r^2 c_i^2 - b_r^2 c_i^2,\;\;&&B_4=-a_i^2 b_r^2 - a_r^2 c_i^2,
    \end{alignat*} 
        and $$D=\{(a_r,b_r,c_i)\in[-1,1]^3\;|\;d_r,e_r,f_r,g_r,h_r\in[-1,1]\}.$$
    $T_{8F}^{(3)}$ is a 3-parameter family of complex Hadamard matrices.
\end{theorem}

\begin{proof}
    One can confirm that the assignment given is a solution for the system of equations in Lemma~\ref{TFlemma2}. Unimodularity is guaranteed when all the real parts get values from $[-1,1]$ because $x_r^2+x_i^2=1$ holds for all of the variables.
\end{proof}

\section{Inequivalence}
\label{ineq_sec}

To show that the introduced families are pairwise inequivalent and not contained in any known families, we provide matrices that belong exclusively to these new families. Often, such considerations are done by utilising invariants such as defect \cite{defOfuni}, fingerprint \cite{fingerprint} or Haagerup set \cite{haagerupSet}. We introduce an algorithm to determine the inclusion precisely. A more detailed study of such algorithms will be conducted in a separate study \cite{upcoming_paper}.

\subsection{Algorithm for determining family membership.} 
\label{alg}

In this section, we describe an algorithm that can be used to determine if an arbitrary matrix $H\in\mathbb{H}(n)$ has an equivalent matrix that belongs to a given family of complex Hadamard matrices. First, we introduce a condition that makes it straightforward to invert a matrix-valued function. We denote the set of the entries of a matrix $A=(a_{i,j})$ by $\{a_{i,j}\}$ and the multiset of entries by $\langle a_{i,j}\rangle$. 

Let $f:\mathbb{T}^k\to M_n(\mathbb{C})$ be a matrix-valued function. We say $f(x_1,\ldots,x_k)$ is \textit{iteratively invertible} if for all $i\in\{1,\ldots,k\}$ there exist an entry $\phi_i(x_1,\ldots,x_i)$ of $f$ such that fixing $x_1,\ldots,x_{i-1}$ makes $\phi_i$ bijective in $x_i$. We refer to such entries $\phi_1,\ldots,\phi_k$ as \textit{iterable entries}
 of $f$.
 
Suppose that $f(x_1,\ldots,x_k)$ is an iteratively invertible function with iterable entries $\phi_1,\ldots, \phi_k$ at indices $(i_1,j_1),\ldots,(i_k,j_k)$. We can now compute $(x_1',\ldots,x'_k)$ from $H=f(x_1',\ldots,x'_k)$ with
\begin{equation}
\label{iter_inv}
(\phi^{-1}_1(h_{i_1,j_1}),\;\phi^{-1}_2(\phi^{-1}_1(h_{i_1,j_1}),h_{i_2,j_2}),\ldots,\phi^{-1}_k(\phi_1^{-1}(h_{i_1,j_1}),\ldots,h_{i_k,j_k})),
\end{equation}
where  $\phi_i^{-1}(x_1,\ldots,x_i)$ denotes the inverse map of $x_i\mapsto \phi_i(x_1,\ldots,x_i)$ for fixed values of $x_1,\ldots,x_{i-1}.$

Let $H^{(k)}_n=f(D)$ be a family of complex Hadamard matrices, where $f:D\to\mathbb{H}_d(n)$ is an iteratively invertible function and $D\subseteq\mathbb{T}^k$. For a given matrix $H\in\mathbb{H}(n),$ we can decide if $H\in H^{(k)}_n$ with the following algorithm.

\begin{enumerate}[label=Step \arabic*., leftmargin=*, align=left]
\setcounter{enumi}{-1}
    \item Choose iterable entries  $\phi_1,\ldots,\phi_k$ of $f$ and compute $\phi^{-1}_1,\ldots,\phi^{-1}_k.$

    \item Find all mutually permutation inequivalent dephased matrices $H'$ equivalent to $H$.
    
\item For each $H' = (h'_{i,j})$, perform a backtracking search to find parameters $x' = (x'_1, \dots, x'_k)$, such that $\langle f_{i,j}(x') \rangle = \langle h'_{i,j} \rangle$, by evaluating the formula~\eqref{iter_inv} with elements from $\{h'_{i,j}\}$.

    \item For each pair $(H', x')$, determine if $f(x'_1,\ldots,x_k')$ and $H'$ are permutation equivalent.
\end{enumerate}

If we find a permutation equivalence in Step 3, we can output permutation matrices $P_1$ and $P_2$ as well as an evaluation point $(x_1'\ldots,x'_k)$ such that $\mathscr{D}(P_1HP_2)=f(x_1',\ldots,x_k')$ implying that $H\in H^{(k)}_n$. If no permutation equivalence is found, then we have exhausted all possibilities and conclude that $H\notin H_{n}^{(k)}$.  Permutation equivalence of two matrices can be determined with an algorithm from \cite{upcoming_paper}.

\begin{example}
    Consider the affine family $F_{4}^{(1)}$ given by 
    $$F_4(a)=
    \begin{bmatrix}
         1 & 1 & 1 & 1 \\
 1 & i a & -1 & -i a \\
 1 & -1 & 1 & -1 \\
 1 & -i a & -1 & i a \\
    \end{bmatrix}.$$
    Suppose that we want to determine whether a complex Hadamard matrix 
    $$
G=\begin{bmatrix}
 \gamma & \beta\gamma & \beta & -1 \\
 \overline{\beta} \gamma & \gamma & -1 & \overline{\beta} \\
 \overline{\beta}  & -1 & \gamma & \overline{\beta} \gamma \\
 -1 & \beta & \beta\gamma & \gamma \\
\end{bmatrix},
    $$
    where $\beta=\exp(\sqrt{-3})$ and $\gamma=\exp(\sqrt{-7})$,
    belongs to $F^{(1)}_4.$ 
    
    The function $F_4(a)$ has only one parameter, so we only need to choose one iterable entry. We can choose $\phi_1(a)=[F_4(a)]_{2,2}=ia$ and obtain its inverse  $\phi_1^{-1}(h)=-ih.$

    In Step 1, we find all the dephased matrices that are equivalent to $G$ up to permutation equivalence. There are only two such matrices:
    $$
   G_1= \begin{bmatrix}
         1 & 1 & 1 & 1 \\
 1 & 1 & -1 & -1 \\
 1 & -1 & \gamma^2 & -\gamma^2 \\
 1 & -1 & -\gamma^2 & \gamma^2 \\
    \end{bmatrix}
    \text{ and } 
      G_2=  \begin{bmatrix}
     1 & 1 & 1 & 1 \\
 1 & 1 & -1 & -1 \\
 1 & -1 & \gamma^{-2} & -\gamma^{-2} \\
 1 & -1 & -\gamma^{-2} & \gamma^{-2} \\
 \end{bmatrix}.
    $$

In Step 2, we check if there exists  $h\in\{1,-1,\gamma^2,-\gamma^2\}$ such that the multiset of the entries of $F_4(\phi_1^{-1}(h))$ equals the multiset of the entries of $G_1$. We find two such matrices $F_4(-i\gamma^2)$ and $F_4(i\gamma^2).$ Similarly, for $G_2$ we find the matrices $F_4(-i\gamma^{-2})$ and $F_4(i\gamma^{-2}).$

Finally, in Step 3, we compute that $F_4(-i\gamma^2)$ and $F_4(i\gamma^2)$ are both permutation equivalent to $G_1$. Also, $F_4(-i\gamma^{-2})$ and $F_4(i\gamma^{-2})$ are both permutation equivalent to $G_2$. We can output any of the found equalities
\begin{align*}
 &   \mathscr{D}(P_1GP_2)=F_4(i\gamma^2),\qquad \mathscr{D}(P_2GP_2)=F_4(-i\gamma^2)\\
   & \mathscr{D}(P_3GP_2)=F_4(i\gamma^{-2})\;\text{ or }\mathscr{D}(P_4GP_2)=F_4(-i\gamma^{-2}),
\end{align*}
where $P_1=(e_1\;e_3\;e_4\;e_2)$, $P_2=(e_1\;e_3\;e_2\;e_4)$, $P_3=(e_4\;e_2\;e_1\;e_3)$, and $P_4=(e_2\;e_4\;e_1\;e_3),$ that all prove the inclusion $G\in F^{(1)}_4.$

When we run the same test with
    $$
G'=\begin{bmatrix}
 \gamma & \beta\gamma & \beta & -1 \\
 \overline{\beta} \gamma & \gamma & -1 & \overline{\beta} \\
 \overline{\beta}  & -1 & \gamma & \overline{\beta} \gamma \\
 -1 & \beta & -\beta\gamma & -\gamma \\
\end{bmatrix},
    $$
    which is not a complex Hadamard matrix, the algorithm proceeds the same way until Step 3. At that point, no satisfying permutations are found, and the algorithm returns a negative result.

In an implementation of the algorithm, a depth-first traversal is more natural, as it allows for immediate termination upon finding a valid solution.

\end{example}

For the family $T^{(3)}_{8E}$, it is particularly easy to find iterable entries as all three parameters appear as matrix entries in $\mathcal{S}_{E}$. Hence, we can take the iterable entries simply as $\phi_1(a)=a$, $\phi_2(b)=b$ and $\phi_3(c)=c$.

We do not need the functions $\phi_i$ to be necessarily bijective in the whole domain $D$. It is enough if we can divide our considerations into a sufficiently small number of cases, usually concerning different branches of the square root. This strengthens the approach as for example the family $T^{(3)}_{8F}$ we cannot find suitable entries that would be bijective on the whole domain.

To apply the algorithm for the $T_{8F}^{(3)}$, we first have to dephase $\mathcal{S}_{F}$. Then we can choose the entries as $\phi_1(a)=[\mathscr{D}(\mathcal{S}_F)]_{8,2}=1/a^2$, $\phi_2(b)=[\mathscr{D}(\mathcal{S}_F)]_{7,2}=1/b^2$ and $\phi_3(a,b,c)=[\mathscr{D}(\mathcal{S}_F)]_{5,3}=1/e'(a,b,c)^2$. For each of these functions we obtain two possible inverses depending on the sign of the square root. We can now determine if $H\in T_{8F}^{(3)}$ by running the algorithm for all 8 different choices of the inverse functions.

Similar to $T^{(3)}_{8E}$ and $T^{(3)}_{8F}$, it is not difficult to find suitable iterable entries for the functions defining $T^{(3)}_{8C}$ and $T^{(3)}_{8D}$. Furthermore, one can utilize the algorithm for all of the families presented in Section~\ref{sec_4_prev_mat}.

\subsection{Inequivalence of the families}
\label{in_seq}

We have implemented the algorithm described in Section~\ref{alg} using floating-point arithmetic. We treat two complex numbers $z_1$ and $z_2$ as equal whenever
$| z_2 - z_1| < \varepsilon$, where $\varepsilon$ is an accuracy parameter given for the algorithm.
Choosing $\varepsilon$ larger than the worst-case floating-point error ensures that non-inclusion results can be regarded as exact.

The following result shows that we have indeed found previously unknown matrices.

\begin{theorem}
\label{inequivalence}
    Let $H^{(k)}_8$ be a $k$-parameter family of complex Hadamard matrices of order 8 from the list
    \begin{equation}
    \label{list}
    \begin{split}
        F^{(5)}_{8},\; D^{(5)}_{8A},\; D^{(5)}_{8B},\; &S_{8A}^{(4)},\; S^{(4)}_{8B}, \\
        T_{8B}^{(3)},\; T_{8C}^{(3)},\; T^{(3)}_{8D},\;(T^{(3)}_{8D})^T,&\; T^{(3)}_{8E},\; (T^{(3)}_{8E})^T,\; T^{(3)}_{8F}.
    \end{split}
    \end{equation}
     Let $f:D\to\mathbb{H}(8)$ be the function defining $H^{(k)}_8$ as given in Sections~\ref{sec_4_prev_mat} and~\ref{results}.
    Consider $H'=f(e^{2i},e^{3i},\dots,e^{(k+1)i})\in\mathbb{H}(8)$. The only family from the list \eqref{list} to which $H'$ belongs is $H_8^{(k)}$ itself.
\end{theorem}

\begin{proof}
    One can check that the point $(e^{2i},e^{3i},\dots,e^{(k+1)i})$ belongs to the domains of all the families in the list \eqref{list}. An exhaustive computer search with the algorithm described in Section~\ref{alg} yields the result.
\end{proof}

An immediate corollary of Theorem~\ref{inequivalence} is that all the families in the list \eqref{list} are pairwise inequivalent, and no family is contained within another. One can also confirm that the point $(e^{2i},e^{3i},\ldots,e^{(k+1)i})$ is not a boundary point of any of the domains, meaning there exists a neighbourhood around that point belonging uniquely to the corresponding family. Furthermore, at the given point, the rank of the Jacobian matrix of the defining function agrees with the number of parameters in each case.

The transposes of $T_{8C}^{(3)}$ and $T_{8F}^{(3)}$ are not included in
Theorem~\ref{inequivalence}, since the algorithm suggests that every matrix in these
families also belongs to the corresponding transposed family. However, due to the complexity of the defining functions, we have not been able to
show the symmetry of the families rigorously.

The systems of polynomial equations in Lemmas~\ref{TClemma}, \ref{TDlemma}, and
\ref{TFlemma2} also admit solutions that differ from the ones presented by the choice of
signs of the square roots. We did not find any matrices arising from these
alternative solutions that would not belong to the corresponding family given by the presented solutions. A proof that the families obtained from the different sign choices are indeed equivalent is left for future work.

\subsection{Butson-type matrices }

The classification of complex Hadamard matrices of Butson-type into equivalence classes has been completed for many small parameter values of $n$ and $q$. These results are extensively documented in \cite{butson}. A study of the Butson matrices suggests that there are still order 8 families to discover. 

In \cite{sölBut}, all BH$(8, 4)$ Butson-type matrices are placed in some affine family from the list \eqref{affine_list}. For BH$(8,6)$ matrices, such a categorisation to known families is not possible as there are matrices that do not belong to any of the families from the list \eqref{list}. Consider a BH$(8,6)$ matrix $B_1=$

$$\text{exp}\bigg(\frac{2 \pi i}{6}
\begin{bmatrix}
0 &  0 &  0 &  0 &  0 &  0 &  0 &  0 \\
 0 &  0 &  0 &  0 &  3 &  3 &  3 &  3 \\
 0 &  1 &  3 &  4 &  0 &  2 &  3 &  5 \\
 0 &  2 &  5 &  3 &  4 &  5 &  2 &  1 \\
 0 &  3 &  2 &  5 &  2 &  5 &  1 &  4 \\
 0 &  3 &  4 &  1 &  5 &  2 &  0 &  3 \\
 0 &  4 &  1 &  3 &  3 &  2 &  5 &  0 \\
 0 &  5 &  3 &  2 &  1 &  5 &  4 &  2 \\    
\end{bmatrix} \bigg),$$
where exp$(\cdot)$ denotes element-wise exponentiation, and its inequivalent transpose $B_2=B_1^T$.

\begin{proposition}
    The matrices $B_1$ and $B_2$ do not belong in any of the families in the list \eqref{list}.
\end{proposition}

\begin{proof}
    This is a result of an exhaustive computer search with the algorithm described in Section~\ref{alg}.
\end{proof}

We can also find Butson matrices with $n=8$ and $q>6$ that do not belong to any of the known families of complex Hadamard matrices.

\section*{Acknowledgements}
I would like to thank Professor Patric Östergård for his insightful guidance and valuable comments. I also thank Wojciech Bruzda for helpful discussions and the anonymous referees for their constructive feedback. The financial support of the Finnish Ministry of Education and Culture through the Quantum Doctoral Education Pilot Program (QDOC VN/3137/2024-OKM-4) and the Research Council of Finland through the Finnish Quantum Flagship project (Aalto 358877) is gratefully acknowledged.

\bibliographystyle{abbrv} 
\bibliography{refs} 
\end{document}